\documentclass[amssymb,amstex,amsmath,10pt,a4paper]{amsart}
\usepackage{amsmath,amssymb,latexsym,amsfonts,url}

\newtheorem{thm}{Theorem}[section]

\newtheorem{df}[thm]{Definition}
\newtheorem{prop}[thm]{Proposition}

\newtheorem{lem}[thm]{Lemma}

\newcommand{\dist}{{\rm dist}}
\newcommand{\area}{{\rm Area}}
\newcommand{\length}{{\rm Length}}
\newcommand{\cv}{{\rm Conv}}

\newcommand{\Div}{{\rm div}}
\newenvironment{pf}{\noindent {\it Proof.}}{\\}
\newenvironment{Rmk}{\noindent {\it Remark.}}{\\}
\numberwithin{equation}{section}
\newcommand{\cone}{\mbox{$\times\hspace*{-0.228cm}\times$}}

\begin{document}
\title[Regularity of soap film-like surfaces spanning graphs]{Regularity of soap film-like surfaces \\spanning graphs in a Riemannian manifold}
\author[R. Gulliver, S.-H. Park, J. Pyo, and K. Seo]{Robert
Gulliver, Sung-ho Park, Juncheol Pyo, and Keomkyo Seo}

\begin{abstract}
Let $M$ be an $n$-dimensional complete simply connected Riemannian
manifold with sectional curvature bounded above by a nonpositive
constant $-\kappa^2$. Using the cone total curvature $TC(\Gamma)$ of a
graph $\Gamma$ which was introduced by Gulliver and Yamada
\cite{GY}, we prove that the density at any point of a soap
film-like surface $\Sigma$ spanning a graph $\Gamma \subset M$ is
less than or equal to $\displaystyle{\frac{1}{2\pi}\{TC(\Gamma) -
\kappa^{2}\area(p\mbox{$\times\hspace*{-0.178cm}\times$}\Gamma)\}}$.
From this density estimate we obtain the regularity theorems for
soap film-like surfaces spanning graphs with small total
curvature. In particular, when $n=3$, this density estimate implies
that if
\begin{eqnarray*}
TC(\Gamma) < 3.649\pi + \kappa^2 \displaystyle{\inf_{p\in M}
\area({p\mbox{$\times\hspace*{-0.178cm}\times$}\Gamma})},
\end{eqnarray*}
then the only possible singularities of a piecewise smooth
$(\mathbf{M},0,\delta)$-minimizing set $\Sigma$ are the
$Y$-singularity cone. 
In a manifold with sectional curvature bounded above by $b^2$ and
diameter bounded by $\pi/b$, we obtain similar results for any
soap film-like surfaces spanning a graph with the corresponding
bound on cone total curvature.  \\

\end{abstract}

\maketitle

\section{Introduction}
In 2002, Ekholm, White, and Wienholtz \cite{EWW} ingeniously
proved that in an $n$-dimensional Euclidean space, classical
minimal surfaces of arbitrary topological type bounded by a Jordan
curve with total curvature at most $4\pi$ must be smoothly
embedded. One year later, Choe and Gulliver \cite{CG} extended the
Ekholm-White-Wienholtz result to minimal surfaces in an
$n$-dimensional complete simply connected Riemannian manifold $M$
with sectional curvature bounded above by a nonpositive constant
$-\kappa^2$. More precisely, they proved that if $\Gamma$ is a
Jordan curve in $M$ with total curvature
\begin{eqnarray*}
TC(\Gamma) := \int_\Gamma |\vec{k}| ds \leq 4\pi +
\kappa^2\inf_{p\in M} \area (p \cone \Gamma),
\end{eqnarray*}
then every branched minimal surface bounded by $\Gamma$ is
embedded. Here $p\cone\Gamma$ is the $geodesic$ $cone$, which is
the union of the geodesic segments $\overline{pq}$ from $p$ to $q$
over all $q\in \Gamma$. Moreover they proved a similar theorem for
minimal surfaces in the hemisphere (\cite{CG}).

For nonclassical minimal surfaces,  Ekholm, White, and Wienholtz
also showed that for a given soap film-like surface $\Sigma$ with
a simple closed boundary curve $\Gamma$ and with density at least
one at every point on the support of $\Sigma \setminus \Gamma$,
the condition $TC(\Gamma) \leq 3\pi$ implies that $\Sigma$ is
smooth in the interior. Recall that the {\it density} of a surface
$\Sigma$ at a point $p\in M$ is defined to be
\begin{eqnarray*}
\Theta (\Sigma, p) = \lim_{\varepsilon \rightarrow 0}\frac{\area
(\Sigma \cap B_\varepsilon (p))}{\pi \varepsilon^2},
\end{eqnarray*}
where $B_\varepsilon (p)$ is the geodesic ball of $M$ with radius
$\varepsilon$, centered at $p$. Recently, Gulliver and Yamada
\cite{GY} proved that similar regularity theorems hold for soap
film-like surfaces spanning graphs in $\mathbb{R}^n$ 

A {\it graph} $\Gamma$ is a finite union of closed arcs $a_i$
meeting at
vertices $q_j$, each of which has valence at least $2$. Here the
{\it valence} of a vertex $q$ is defined by the number of times
$q$ occurs as an end point among all of the $1$-simplices $a_i$.
We assume that each $a_i$ is $C^2$ and meets its end points with
$C^1$ smoothness. As in \cite{GY}, we define the {\it
contribution to cone total curvature} $tc(q)$ at a vertex $q$ by
\begin{eqnarray*}
tc(q):=\sup_{e\in T_{q}(M),\|e\|=1}\Big\{\sum_{a_{k}:q\in
a_{k}}\Big(\frac{\pi}{2}-\angle_{q}(T_{k}(q),e)\Big)\Big\},
\end{eqnarray*}
where $\angle_{q}(T_{k}(q),e)$ is the angle between the tangent
vector $T_{k}(q)$ pointing into $a_k$ at $q$ and the direction $e$.
If
$\Gamma$ is a piecewise smooth Jordan curve, then the contribution
to cone total curvature $tc(q)$ at $q$ is nothing but the exterior
angle at the vertex $q$: in fact, the supremum is assumed for $e$
in the shorter great-circle arc of $S^{n-1}$
between the two tangent vectors at $q$. Define
the {\it cone total curvature} $TC(\Gamma)$ of a graph $\Gamma$ by
\begin{eqnarray*}
TC(\Gamma):=\int_{\Gamma^{\rm reg}}|\overrightarrow{k}|ds \ + \
\sum_{q \in V(\Gamma)} tc(q),
\end{eqnarray*}
where $V(\Gamma) :=\{ {\rm vertices \ of \ \Gamma} \}$ and
$\Gamma^{\rm reg} := \Gamma \setminus V(\Gamma)$ .\\[2mm]

In soap film experiments, one can
construct soap films spanning graphs so that three surfaces meet
along an edge or six surfaces meet at a point. Indeed there are
only two possible singularities on area-minimizing soap film-like
surfaces in $\mathbf{R}^3$ (see \cite{AT}). The tangent cone to
a soap film-like surface $\Sigma$ at a singularity is an
area-minimizing cone: either the $Y$-singularity cone, which
consists of three
half-planes meeting at $120^\circ$, or the $T$-singularity cone
which is spanned by the regular tetrahedron with vertex at its
center of mass. We will denote by
$C_Y = 3/2$ the density at its vertex of the $Y$-singularity cone
and by $C_T = \frac{3}{\pi} \arccos (-\frac{1}{3}) \approx 1.8245$
the density at its vertex of the $T$-singularity cone.

Introducing a notion of {\em cone total curvature} $TC(\Gamma)$
for a graph $\Gamma$, that is, a finite $1$-dimensional curved
polyhedron, Gulliver
and Yamada proved that for a soap film-like surface $\Sigma$
spanning a graph $\Gamma$ in $\mathbb{R}^n$, if $TC(\Gamma) \leq
2\pi C_Y = 3\pi$, then $\Sigma$ is an embedded smooth surface or a
subset of the $Y$-singularity cone with planar faces. Moreover
they proved that if $TC(\Gamma) \leq 2\pi C_T$, then a soap
film-like surface $\Sigma$ spanning $\Gamma$ in $\mathbb{R}^3$ has
possibly $Y$-singularities but no other singularities unless it is
a subset of the $T$-singularity cone with planar faces. In this
paper we extend the Gulliver-Yamada results to soap film-like
surfaces spanning graphs in an $n$-dimensional Riemannian manifold
$M$ with an upper bound on sectional curvature.

Let $\mathcal{S}_\Gamma$ be the collection of all immersed images
$\Sigma = \cup \Sigma_i$
of a finite union of $C^2$-smooth open two-dimensional manifolds
$\Sigma_i$ with compact closure, of class $C^1$ up to the
piecewise $C^1$ boundary $\partial \Sigma_i$, so that
$\Gamma \subset \cup \partial \Sigma_i$. Throughout this paper,
we consider singular surfaces spanning an embedded graph $\Gamma$
in the class $\mathcal{S}_\Gamma$.  In particular,
$\Sigma \in \mathcal{S}_\Gamma$ is a rectifiable varifold 
\cite{Almgren}.  A surface $\Sigma$ in $\mathcal{S}_\Gamma$
is said to be {\it strongly stationary with respect to} $\Gamma$
if the first variation of the area of $\Sigma$ is at most equal to
the integral over $\Gamma$ of the length of the orthogonal component
of the variational vector field normal to
$\Gamma$. (See Definition \ref{def:strongly stationarity} below
and also \cite{EWW}, \cite{GY}.) Unfortunately one cannot
distinguish the boundary of a singular surface $\Sigma \in
\mathcal{S}_\Gamma$ from $\cup \partial \Sigma_i$ in a classical
sense. Moreover the definition of boundary motivated by Stokes'
Theorem \cite{Federer} is no longer valid (other than {\it modulo} 
two) because $\Sigma$ is not
orientable in general. Therefore we define the boundary of
$\Sigma\in \mathcal{S}_\Gamma$ in variational terms. (See
Definition \ref{def:variational boundary} below.)

Let $M$ be an $n$-dimensional complete simply connected Riemannian
manifold with sectional curvature bounded above by a nonpositive
constant $-\kappa^2$. We shall prove that if $\Gamma$ is an
embedded graph in $M$ with
\begin{eqnarray*}
TC(\Gamma) \leq 3\pi + \kappa^2 \Big(\inf_{p\in M}
\area({p\cone\Gamma})\Big),
\end{eqnarray*}
then a soap film-like surface $\Sigma$ spanning $\Gamma$ is an
embedded surface or a subset of the piecewise totally geodesic
$Y$-singularity cone (Theorem
\ref{thm:3pi}). More precisely, the infimum of $\area (p\cone
\Gamma)$ may be taken only over $p$ in the geodesic convex hull
$\cv (\Gamma)$ of
$\Gamma$. Furthermore, we shall prove that if $\Gamma$ is a graph
in $M^3$ with
\begin{eqnarray*}
TC(\Gamma) \leq 2\pi C_T+ \kappa^2 \Big(\inf_{p\in \cv(\Gamma)}
\area({p\cone\Gamma})\Big),
\end{eqnarray*}
then the only possible singularities of a soap film-like surface
$\Sigma$ in $M^3$ are the $Y$-singularity cone or, in special
cases, the
$T$-singularity cone (Theorem \ref{thm:2pi}). Similar results for
a manifold with bounded diameter, and the corresponding
positive upper bound on
sectional curvature, will be demonstrated in Section 4.

The key steps of the proofs of our theorems are as follows. We
first compare the density $\Theta (\Sigma, p)$ of $\Sigma$ in
$\mathcal{S}_\Gamma$ with the density $\Theta (p\cone\Gamma, p)$
of the cone $p\cone \Gamma$ for $p\in \Sigma\setminus \Gamma$.
This is related to the monotonicity property of minimal surfaces.
Next, we construct a certain $2$-dimensional cone with constant
curvature metric locally isometric to
the simply connected space form $\overline{M}^2$ of
constant sectional curvature, which corresponds to
$p\cone\Gamma$ in $M$. With comparison results for geodesic
curvature of cones and the Gauss-Bonnet
Theorem for the cone, we shall obtain the above theorems.

\section{Preliminaries}
Let $\Gamma$ be a graph in an $n$-dimensional Riemannian manifold
$M$, and let $\mathcal{S}_\Gamma$ be the class of singular
surfaces as described in the introduction.
\begin{df}[\cite{EWW}] \label{def:strongly stationarity}
{\rm A rectifiable varifold $\Sigma$ in $M$ is called }strongly
stationary {\rm with respect to $\Gamma$ if for all smooth
variations $\phi : \mathbb{R} \times M \rightarrow M$ with $\phi(0,
x) \equiv x$ we have
\begin{eqnarray*}
\frac{d}{dt}\Big|_{t=0} \Big( \area (\phi(t, \Sigma)) + \area
(\phi([0,t]\times \Gamma))\Big) \geq 0 .
\end{eqnarray*}
}
\end{df}
Note that the strong stationarity of $\Sigma$ implies that $\Sigma
\in \mathcal{S}_\Gamma$ is stationary. In fact, the first variation
formula for  area \cite{Simon} gives
\begin{eqnarray*}
\frac{d}{dt}\Big|_{t=0} \area (\phi(t, \Sigma)) = -\int_{\Sigma}
\langle \overrightarrow{H}, X^{\perp}\rangle dA + \int_{\cup
\partial \Sigma_i} \sum_{j \in J(p)} \langle \nu_j (p), X^{\perp}
(p) \rangle ds,
\end{eqnarray*}
where $\overrightarrow{H}$ is a mean curvature vector of $\Sigma$;
$X^{\perp}$ is the normal component of the variational vector
field $X=\frac{\partial \phi }{\partial t} (0,x)$; $J(p)$ indexes
the collection of surfaces $\Sigma_j$ which meet at a point $p$ in
$\cup \partial \Sigma_i$; and $\nu_j(p)$ is the
outward unit conormal vector to $\partial \Sigma_j$ along 
$\cup \partial \Sigma_i$ where surfaces $\Sigma_j$ meet, 
for $j\in J(p)$.

If we take the variational vector field
$X=\frac{\partial \phi }{\partial t} (0,x)$
to be supported in $M\setminus \cup \partial \Sigma_i$,
then the stationarity of $\Sigma$ implies that the mean curvature
vector
$\overrightarrow{H}\equiv \overrightarrow{0}$ on the interior of
each $\Sigma_i$. Furthermore, if we choose $X$ supported away from
$\Gamma$, then it follows that
\begin{eqnarray} \label{balancing}
\nu_{\Sigma}(p):= \sum_{j\in J(p)\subset
I}\nu_{\Sigma_{j}}(p)=\overrightarrow{0}
\end{eqnarray}
for almost all $p$ on $\cup \partial \Sigma_i \setminus \Gamma$,
since the choice of $X$ is arbitrary along
$\cup \partial \Sigma_i \setminus \Gamma$. We call the equation
(\ref{balancing}) the {\it balancing}
of $\nu_{\Sigma_i}$ along the singular curves of $\Sigma$ away
from $\Gamma$ \cite[p.\ 319]{GY}.

It should be mentioned that the strong stationarity with respect
to $\Gamma$ for a surface in $\mathcal{S}_\Gamma$ is equivalent to
stationarity in $M\setminus \Gamma$ plus the following boundary
condition.
\begin{df}[\cite{EWW}, \cite{GY}] \label{def:variational boundary}
{\rm $\Gamma$  is said to be the} variational boundary {\rm of a
surface $\Sigma$ if there exists an $\mathcal{H}^1$ measurable
vector field $\nu_\Sigma$ along $\Gamma$ which is orthogonal to
$\Gamma$, with $|\nu_\Sigma| \leq 1$ a.e., such that for all smooth
vector fields $X$ defined on $M$,
\begin{eqnarray*}
\int_{\Sigma} \Div X^T \, dA = \int_\Gamma \langle X, \nu_{\Sigma}
\rangle \, ds .
\end{eqnarray*}
}
\end{df}
Another good mathematical model for soap films is an
$(\mathbf{M},\varepsilon,\delta)$-minimal set, as was introduced
by F. Almgren \cite{Almgren}.
\begin{df}[\cite{Almgren}]
{\rm Let $\varepsilon$ be a function $\varepsilon(r) = Cr^\alpha$
for some $0\leq C <\infty$, $0<\alpha<1/3$ and $\delta>0$.
$\Sigma \subset M$ is said to be} an
$(\mathbf{M},\varepsilon,\delta)$-minimal set {\rm with respect to
$\Gamma \subset M$ if $\Sigma$ is an $m$-dimensional rectifiable
set and if, for every Lipschitz mapping $\phi : M \rightarrow M$
with the diameter $r={\rm diam}(W \cup \phi(W)) < \delta$,
\begin{eqnarray*}
\mathcal{H}^m(\Sigma \cap W) \leq (1 +
\varepsilon(r))\,\mathcal{H}^m ( \phi(\Sigma \cap W)),
\end{eqnarray*}
where $W = \{x : \phi (x) \neq x \}$. }
\end{df}

The following theorems for soap film-like surfaces in
$\mathbb{R}^3$ are due to Jean Taylor.
%
\begin{thm}[\cite{Taylor}] \label{Taylor:cone}
The tangent cone of an $(\mathbf{M},0,\delta)$-minimal set
$\Sigma$ at $p \in \Sigma \setminus \Gamma$ is
area minimizing with respect to the
intersection with the unit sphere at $p$; Moreover, the plane, the
$Y$-cone and the $T$-cone are the only possibilities for
minimizing cones.
\end{thm}
%

%
\begin{thm}[\cite{Taylor}] \label{Taylor:regularity}
An $(\mathbf{M},\varepsilon,\delta)$-minimal set with respect to
$\Gamma$ consists of $C^{1,\alpha}$ surfaces meeting smoothly in
threes at $120^{\circ}$ angles along smooth curves, with these
curves meeting in fours at angles of $\arccos (-1/3)$, away from
$\Gamma$.
\end{thm}
It was later proved by Kinderlehrer, Nirenberg and Spruck that the
surfaces which comprise $\Sigma$ are as smooth as the ambient
manifold \cite{KNS}.

%
%
\section{Regularity of soap film-like surfaces in negatively curved spaces}
Let $M$ be an $n$-dimensional complete simply connected Riemannian
manifold with sectional curvature bounded above by a nonpositive
constant $-\kappa^2$. In this section we shall derive regularity
theorems for soap film-like surfaces in $M$. The key step in the
extension of the theorems to the variable curvature ambient space
is to carry the data of $p\cone \Gamma$ over to the simply
connected space form $\overline{M}$ of constant sectional
curvature $-\kappa^2$, where $\kappa \geq 0$. To do this  we
construct a constant curvature metric on $p\cone \Gamma$. For the
sake of clarity, we give the definition as follows.

\begin{df}[\cite{Choe}, \cite{CG}]  \label{eq:q}
Let $\Gamma$ be an immersed piecewise $C^1$ curve in $M$.
Let $\widehat{g}$ be a new metric on $p\cone \Gamma$ with constant
Gauss curvature $-\kappa^{2}$, such that the distance from $p$
remains the same as in the original metric $g$, and so does the
arclength element of $\Gamma$.
\end{df}
We shall construct the new metric $\widehat{g}$ explicitly, and
then extend definition \ref{eq:q} to the case where $\Gamma$ is a
graph.  Assuming for the moment that $\Gamma$ is a curve and that
$\widehat{g}$ exists, one can see from the above definition that
every geodesic from $p$ under $g$ remains a geodesic of equal
length under $\widehat{g}$, the length of any arc of $\Gamma$
remains the same, and the angles between the tangent vector to
$\Gamma$ and the geodesic from $p$ remain unchanged.  Given a cone
$C:= p\cone \Gamma$, we shall denote by $\widehat{C}$ the
two-dimensional Riemannian manifold $(C, \widehat{g})$, which will
be singular at $p$.

Let $r(s)$ be the distance in $C$ from $p$ to the corresponding
point of $\Gamma$ for an arc-length parameter $s$ of $\Gamma$.
Then choose a point $\widehat{p} \in \mathbb{H}^2(-\kappa^2)$, and
let a curve $\widehat{\Gamma}$ locally isometric to $\Gamma$ be
traced out in $\mathbb{H}^2(-\kappa^2)$ so that the distance from
$\widehat{p}$ equals $r(s)$. $\widehat\Gamma$ will close up, to
form $\widehat\Gamma$ as a simply closed curve, in some Riemannian
covering space $\widehat M^2$ of
$\mathbb{H}^2(-\kappa^2)\setminus\widehat p$. Then $\widehat{C}$
can be written as $\widehat{p}\cone \widehat{\Gamma}$ in
$\widehat M$, along with the metric of $\widehat M$, which is
singular at $\widehat p$. (See \cite{Choe} for more details.)

To prove the next proposition, we need the following well-known
fact due to Euler. (See \cite{Ore}.)
%
\begin{lem}[Euler] \label{Euler}
For a connected graph $\Gamma$ with even valence at each vertex,
there is a continuous mapping of the circle to $\Gamma$ which
traverses each edge exactly once.
\end{lem}

Let $\Gamma'$ be the double covering of the graph
$\Gamma = \cup_{i=1}^m c_i$. By Lemma \ref{Euler}, we may choose
an ordering for each edge such that $\Gamma'$ is a piecewise smooth
immersion of $\mathbb{S}^1$. In other words,
$\Gamma'=\cup_{j=1}^{2m} c'_{j}$, where each $c_i$ arises twice
as one of the $c_j'$ for $i=1,\cdots,m$. Then Definition
\ref{eq:q} may be applied to $\Gamma'$, to construct a metric
$\widehat g$ of constant Gauss curvature on the disk
$C'=p\cone\Gamma'$. Then each smooth surface $A_i=p\cone c_i$
of the cone $C=p\cone \Gamma$ is covered by two smooth ``fans"
$p\cone c_j'$ and $p\cone c_k'$ of $C'$ ($1\leq j,k\leq 2m$).
But $c_j'$ and $c_k'$ are copies of the arc $c_i$ of $\Gamma$, so
$p\cone c_j'$ is isometric to $p\cone c_k'$ using the metric
$\widehat g$ of constant Gauss curvature for both. Thus,
$p\cone c_i$ inherits the metric $\widehat g$ from either
$p\cone c_j'$ or $p\cone c_k'$, to form the singular cone
$\widehat C$ with constant Gauss curvature on the interior of
$p\cone c_i$, $1\leq i\leq m$, and with a singular curve
$\overline{pq}$ for each vertex $q$ of $\Gamma$. This completes
the extension of Definition \ref{eq:q} for any graph $\Gamma$.

\begin{prop} \label{prop:density in M}
Let $M$ be an $n$-dimensional simply connected Riemannian manifold
with sectional curvature $K_M \leq -\kappa^2 \leq 0$, and let
$\Sigma \in \mathcal{S}_\Gamma$ be a strongly stationary surface
in $M$. Then we have, for $p \in \Sigma \setminus \Gamma,$
$$\Theta(\Sigma,p) < \Theta(\widehat{C},p),$$
unless $\Sigma$ is a cone over $p$ with totally geodesic faces
of constant curvature $-\kappa^{2}$.
\end{prop}
\begin{proof}
First let us assume $\widehat K \equiv -\kappa^2 < 0$; the case
$\kappa=0$ will be treated similarly. Denote $r(x):=\dist(p,x)$,
the distance function in $M$ from $p\in M$. Let
$G(r(x)):=\log \tanh(\kappa r(x)/2)$ be the Green's function
for the two-dimensional
hyperbolic plane $\mathbb{H}^{2}(-\kappa^{2})$ with Gauss
curvature $-\kappa^{2}$. On an immersed minimal surface $\Sigma_i$
in $M$, it follows from \cite{CG} that
%
\begin{eqnarray}\label{green ftn on M}
\triangle_{\Sigma}G(r)=2\kappa^{2}\frac{\cosh \kappa r}{\sinh^{2}
\kappa r} (1-|\nabla_{\Sigma_{i}}r|^2) \geq 0.
\end{eqnarray}
Recall that $\Sigma = \cup_{i \in I} \Sigma_i$ and each $\Sigma_i$ is
an immersed minimal surface. Integrating \eqref{green ftn on M} over
$\Sigma_i \setminus B_\varepsilon(p)$ gives
\begin{eqnarray*}
0\leq \int_{\Sigma_i \setminus
B_\varepsilon(p)}\triangle_{\Sigma_{i}}G\,dA
=\int_{\partial(\Sigma_i \setminus B_\varepsilon(p))}
\frac{\kappa}{\sinh \kappa r}\frac{\partial r}{\partial \nu_{\Sigma_i}}ds.
\end{eqnarray*}
Write $\Gamma=\cup_{j=1}^{m} c_{j}$. For a sufficiently small
$\varepsilon >0$, let $C_{\varepsilon}:=C\setminus
B_{\varepsilon}(p) = (p\cone \Gamma)\setminus B_{\varepsilon}(p) $
and $\Sigma_{i,\varepsilon}:=\Sigma_{i} \setminus
B_{\varepsilon}(p)$. The divergence theorem yields
\begin{eqnarray*}
0\leq
\int_{\Sigma_{i,\varepsilon}}\triangle_{\Sigma_{i,\varepsilon}}G\,dA
=\int_{\partial\Sigma_{i,\varepsilon}}\frac{\kappa}{\sinh \kappa
r}\frac{\partial r}{\partial \nu_{\Sigma_{i,\varepsilon}}}ds.
\end{eqnarray*}
Each boundary $\partial\Sigma_{i,\varepsilon}$ consists of three
parts
$$\partial\Sigma_{i,\varepsilon}=(\partial\Sigma_{i}\cap\Gamma)\cup
(\partial B_{\varepsilon}(p)\cap \Sigma_{i})\cup (\partial
\Sigma_{i}\setminus (\Gamma\cup \overline{B_{\varepsilon}(p)})).$$
Summing over $i\in I$, we get
$$0\leq \int_{\Gamma}\frac{\kappa}{\sinh \kappa r}\frac{\partial r}{\partial\nu_{\Sigma}}ds
+\int_{\Sigma \cap \partial B_{\varepsilon}(p)}
\frac{\kappa}{\sinh\kappa\varepsilon}\frac{\partial r}{\partial\nu_{\Sigma}}ds
+\sum_{i\in I}\int_{\partial\Sigma_{i}
\setminus(\Gamma\cup\overline{B_{\varepsilon}(p)})}
\frac{\kappa}{\sinh \kappa r}
\frac{\partial r}{\partial\nu_{\Sigma_{i}}}ds$$
where $\nu_{\Sigma}$ is the
outward unit conormal along the $\partial\Sigma_{\varepsilon}.$
Since $\Sigma \in \mathcal{S}_{\Gamma}$, the last term vanishes
by the balancing condition (\ref{balancing}). Along
$\Sigma\cap \partial B_{\varepsilon}(p)$,
$\frac{\partial r}{\partial\nu_{\Sigma}}\rightarrow -1$ uniformly as
$\varepsilon\rightarrow0$ and hence the second term converges to
$$\lim_{\varepsilon \rightarrow 0}-\kappa \frac{\length (\Sigma \cap \partial
B_{\varepsilon}(p))}{\sinh\kappa\varepsilon}= -2\pi\Theta(\Sigma,p).$$
Therefore we obtain
\begin{equation}\label{eq:a1}
2\pi\Theta(\Sigma,p)\leq
\int_{\Gamma}\frac{\kappa}{\sinh\kappa r}
\frac{\partial r}{\partial\nu_{\Sigma}}ds.
\end{equation}

We repeat the same argument for $\widehat{C} = \widehat{p}\cone
\widehat{\Gamma}$ with $\widehat{C}$ instead of $\Sigma$ and
$A_j=\widehat{p}\cone \widehat{c_{j}}$ instead of
$\Sigma_i$. Thus the cone $\widehat{C}$ is the union of
$A_j$, $1\leq j \leq m$. Since the Green's function
$G(r)$ on $A_j$ satisfies \cite[Lemma 3]{CG}
\begin{eqnarray*}
\triangle_{A_j}G(r)=
2\kappa^{2}\frac{\cosh \kappa r}{\sinh^{2}\kappa r}(1-|\nabla_{N}r|^{2}),
\end{eqnarray*}
it follows that $\triangle_{\widehat{C}}G(r)=0$. Applying the
divergence theorem, we get
$$0= \int_{A_j\setminus
B_{\varepsilon}(p)}\triangle_{\widehat{C}}G\,dA
=\int_{\partial(A_j\setminus B_{\varepsilon}(p))}
\frac{\kappa}{\sinh\kappa r}
\frac{\partial r}{\partial\nu_{\widehat{C}}} ds.$$
Each boundary
$\partial (A_j\setminus B_{\varepsilon}(p))$
consists of three parts:
\begin{eqnarray*}
\partial(A_j \setminus
B_\varepsilon(p))=(\partial A_{j}\cap\Gamma)\cup (\partial
B_{\varepsilon}(p)\cap A_{j})\cup (\partial A_{j}\setminus
(\Gamma\cup \overline{B_{\varepsilon}(p)})).
\end{eqnarray*}
Summing the above equation over $j=1,\cdots,m$, we obtain
$$0=\int_{\widehat{\Gamma}}\frac{\kappa}{\sinh\kappa r}
\frac{\partial r}{\partial\nu_{C}}ds
+\int_{\widehat{C}\cap\partial B_{\varepsilon}(p)}
\frac{\kappa}{\sinh\kappa\varepsilon}
\frac{\partial r}{\partial\nu_{\widehat{C}}}ds
+\sum_{j}\int_{\widehat{p}\cone(\partial\widehat{c_{j}})\setminus
B_{\varepsilon}(p)}\frac{\kappa}{\sinh\kappa r}
\frac{\partial r}{\partial\nu_{A_j}}ds,$$
where $\nu_{\widehat{C}} =\nu_{A_j}$
along $\partial A_j$ and
$\nu_{A_j}$ is the outward unit conormal vector along
the boundary $\partial A_j$. Since the conormal vector
$\nu_{A_j}$ and $\overline{\nabla} r$ are
perpendicular along $\widehat{p}\cone(\partial \widehat{c}_{j})$,
the last term vanishes. As $\varepsilon \rightarrow 0$, the second
term converges to
\begin{eqnarray*}
\lim_{\varepsilon \rightarrow 0}-\kappa \frac{\length
(\widehat{C}\cap B_{\varepsilon}(p))}{\sinh
\kappa{\varepsilon}}=-2\pi \Theta(\widehat{C},p).
\end{eqnarray*}
 Hence we have
\begin{eqnarray} \label{eq:a2}
 2\pi\Theta(\widehat{C},p)=
\int_{\widehat{\Gamma}} \frac{\kappa}{\sinh \kappa
r}\frac{\partial r}{\partial\nu_{\widehat{C}}}ds.
\end{eqnarray}
On the other hand, from the definition of $\widehat{C}$ it is easy
to see that
\begin{eqnarray} \label{eq:widehat C}
\frac{\partial r}{\partial \nu_{C}}= \frac{\partial
r}{\partial\nu_{\widehat{C}}}.
\end{eqnarray}
 Since $\frac{\partial r}{\partial \nu_{C}}\geq\frac{\partial
r}{\partial\nu_{\Sigma}}$ almost everywhere along $\Gamma$, the
inequality (\ref{eq:a1}) and the equations (\ref{eq:a2}),
(\ref{eq:widehat C}) imply that
$$2\pi\Theta(\Sigma,p)\leq
\int_{\Gamma}\frac{\kappa}{\sinh\kappa r}
\frac{\partial r}{\partial\nu_{\Sigma}}ds\leq
\int_{\widehat{\Gamma}}\frac{\kappa}{\sinh\kappa r}
\frac{\partial r}{\partial\nu_{\widehat{C}}}ds
=2\pi\Theta(\widehat{C},p).$$
If equality holds, then $\triangle_{\Sigma} G(r) \equiv0$, which
requires $|\nabla_{\Sigma}r|\equiv1$. This can happen only when
$\Sigma$ is totally geodesic. Similarly, using $G(r)=\log r$, we
can prove the case where $\kappa = 0$.
\end{proof}

The following proposition was stated, but not proved, in Remark 3
of \cite{CG}. The version we will need in this section already
appears as Proposition 4 of that paper; in the following section,
we shall require this more general version.
%
\begin{prop} \label{prop:CG1}
Let $g$ and $\widehat g$ be two continuous, piecewise $C^2$
metrics on the two-dimensional disk C, possibly singular at a
certain point $p \in C$, such that the geodesics through $p$ (the
``radial" geodesics) with respect to $g$ and $\widehat g$ are the same,
with the same arc-length parameter along each radial geodesic. Assume
that the arc length along the Lipschitz continuous boundary
$\Gamma = \partial C$ is the same with respect to both metrics.
For both metrics, assume there are no conjugate points along any
radial geodesic. If the Gauss curvatures satisfy
$\widehat K \geq K$ at each point of $C$, then the inward geodesic
curvatures satisfy $\widehat k \leq k$ at each point of $\Gamma$.
\end{prop}
\begin{proof}
As in Proposition 4 of \cite{CG}, we consider normal Jacobi fields
$V$ and $\widehat V$ along a unit-speed geodesic $\gamma(t)$ from
$p = \gamma(0)$ to $q \in \Gamma$, such that
$V(p) = \widehat V(p)=0$ and $V(q) = \widehat V(q)$. This
construction is possible since there are no conjugate points along
$\gamma$ for either metric. Write $f(t) = |V(\gamma(t))|_g$ and
$\widehat f(t) = |\widehat V(\gamma(t))|_{\widehat g}$: both are
positive except at $t=0$. At $q \in \Gamma$, the metrics are the
same, so we have $|V(q)|_g = |\widehat V(q)|_{\widehat g} = 1.$
Then the inward curvatures $k_0$ and $\widehat k_0$ of the
geodesic circle
centered at $p$ satisfy
$$ k_0(\gamma(t))= \frac{f'(t)}{f(t)} {\rm \quad and \quad}
\widehat k_0(\gamma(t))= \frac{\widehat f'(t)}{\widehat f(t)}. $$
The Jacobi equation $f''(t)+K(\gamma(t))f(t)=0$ implies the Riccati
equation $k_0'+k_0^2 = -K$, and similarly
$\widehat k_0'+{\widehat k_0}^2 = -\widehat K$. Since
$K\leq \widehat K$, these imply
$$ (k_0-\widehat k_0)' + (k_0+\widehat k_0)(k_0-\widehat k_0) \geq
0,$$
with $k_0-\widehat k_0 \to 0$ at the point $\gamma(t)$ as $t\to 0$.
Therefore
$k_0 \geq \widehat k_0$ everywhere along $\gamma$. Observe that
the angles formed by $\Gamma$ and $\gamma$ with respect to $g$ and
$\widehat g$ are the same. It now follows that the inward geodesic
curvatures $k$ and $\widehat k$ of $\Gamma$ with respect to $g$ and
$\widehat g$ satisfy $\widehat k \leq k$. (See the proof of Proposition
4 of \cite{CG}.)
\end{proof}

%
\begin{prop}\cite[Proposition 5]{CG} \label{prop:CG2}
Let $\Gamma$ be a $C^{2}$ curve in $M$ and let $C$ be the cone
$p\cone\Gamma$. If $\widehat{C}$ is the cone $C$ equipped with the
constant curvature metric $\widehat{g}$, as in Definition
\ref{eq:q} above, then $\Theta(C,p)\leq\Theta(\widehat{C},p)$ and
$\area(C)\leq\area(\widehat{C})$.
\end{prop}
Now we prove the following Gauss-Bonnet formula for
two-dimensional cones in a nonpositively curved manifold.
%
\begin{prop}[Gauss-Bonnet formula] \label{prop:GB for widehat of C}
Let $p$ be a point in $M\setminus\Gamma$ for a graph
$\Gamma$
with edges $c_1,\dots,c_m$.  Let 
$\widehat{C}=(C, \widehat{g})$ for $C=p\cone\Gamma$.  Then we have
$$2\pi\Theta(\widehat{C},p)+
\kappa^{2} \area(\widehat{C})=
-\sum_{i=1}^m\int_{\widehat{c_i}} \widehat{k}\, ds +
\sum_{i=1}^m\sum_{j=0,1}\Big(\frac{\pi}{2}-
\angle_{q^i_j}(T_{i}(q^i_j),\overline{q^i_j p})\Big),$$
where $\widehat{k}$ is the geodesic curvature of $\Gamma$ in
$\widehat{C}$ and
$\angle_{q^i_j}(T_{i}(q^i_j),\overline{q^i_j p})$ is
the angle at $q_j^i$ between the tangent vector $T_{i}(q^i_j)$
to $c_i$ and the geodesic $\overline{q^i_j p}$.
\end{prop}

\begin{proof}
Suppose first that $\Gamma$ is a smooth closed curve in $M^n$ but
not necessarily simple, and $p$ is a point of
$M^n \setminus \Gamma$. Choose a sufficiently small
$\varepsilon>0$ such that $\overline{B_{\varepsilon}(p)}$ does not
intersect $\Gamma$. Let
$A=\widehat C\setminus B_{\varepsilon}(p)$ be the annular
region between $\Gamma$ and
$\partial B_{\varepsilon}(p)\cap \widehat C$. The Gauss-Bonnet
formula says that
$$\int_{A}\widehat K\,dA-
\int_{\Gamma}\overrightarrow{k}\cdot\nu_{C}\,ds
-\int_{\partial B_{\varepsilon}(p)\cap
(p\cone\Gamma)}\overrightarrow{k}\cdot\nu_{\widehat C}\,ds
=2\pi\,\chi(A),$$
where $\chi(A)$ is the Euler characteristic of $A$ and
$\widehat K=-\kappa^2$ is the intrinsic Gauss curvature of
$\widehat C$.
Since $A$ is an immersed annulus, we have $\chi(A)=0$.
Along $\partial B_{\varepsilon}(p)\cap \widehat C$, we have
$\nu_{\widehat C}=-\overline{\nabla}r$ and
$\overrightarrow{k}\cdot\nu_{\widehat C}\equiv 
\widehat k_0(\varepsilon)$,
where
$k_0(\varepsilon)=\kappa\cot\kappa\varepsilon$ is its geodesic
curvature,
as in Proposition \ref{prop:CG1}, and $r$ is the distance from
$\widehat p$ in $M$.  Note that
\begin{eqnarray*}\label{eq:a}
\lim_{\varepsilon\rightarrow0}\int_{\partial B_{\varepsilon}(p)\cap
\widehat C}\overrightarrow{k}\cdot\nu_{\widehat C}
 &=&\lim_{\varepsilon\rightarrow0}
\widehat k_0(\varepsilon)\,\length(\partial B_{\varepsilon}(p)\cap
\widehat C)\nonumber\\
 &=&2\pi\Theta(\widehat C,p).\nonumber\\
\end{eqnarray*}
 Letting
$\varepsilon \rightarrow 0$, we get
%
\begin{equation}\label{GB for C hat}
-\kappa^2 \area(\widehat C)
-\int_{\Gamma}\overrightarrow{k}\cdot\nu_{\widehat C}
-2\pi\Theta(\widehat C,p)=0.
\end{equation}

Finally, in the general case of of a {\em graph} $\Gamma\subset M^n$,
we apply Lemma \ref{Euler} to show that the double cover of
$\Gamma$ can be considered as a piecewise smooth immersion
$\Gamma'$ of $\mathbb{S}^1$.

Since $\Gamma'$ is a piecewise-smooth immersion of the circle, we
may apply equation \eqref{GB for C hat} to $\Gamma'$, and conclude
that
$$\kappa^2\area({\widehat C}')+
\int_{\Gamma'}\overrightarrow{k}\cdot\nu_{C}\,ds+
2\pi\Theta({\widehat C}',p)=0.
$$

Let $q^i_0$, $q^i_1$ be the end points of the smooth segment
$c_i$ for $i=1,\cdots,m$. We denote $q^i_j\sim q^k_\ell$
if they represent the same point of $M$ where $c_i$ and $c_k$ meet.
Denote $A'_j = \widehat{p}\cone{c}'_j$ with the
metric of constant curvature. Then we see that
$\widehat{C}'=\cup_{j=1}^{2m} A'_j =\widehat{p}\cone \Gamma'$.

Similar arguments as in the above smooth case show that
\begin{equation}\label{big3}
2\pi\Theta(\widehat C',p)=-\kappa^2\area(A)-
\sum_{i=1}^{m}\int_{c_{i}} \overrightarrow{k}\cdot\nu_{c_{i}}\,ds
+\sum_{i=1}^{m}\sum_{\ell=0,1}(\frac{\pi}{2}-
\angle_{q^{i}_{\ell}}(T_{i}(q^{i}_{\ell}),\overline{q^{i}_{\ell}p})).
\end{equation}

Indeed, the last term is equal to the sum over vertices of the sum
of the exterior angles of the piecewise smooth curve  $\Gamma'$ at
the vertex $q^{i}_{1}\sim q^{k}_{0}$.  To see this, suppose that
$c_{i}$ and $c_{k}$ are the consecutive edges in $\Gamma$ joined
at $q^{i}_{1}\sim q^{k}_{0}$; then
$$\Big(\frac{\pi}{2}-
\angle_{q^i_1}(T_i(q^i_1),\overline{q^i_1 p})\Big)+
\Big(\frac{\pi}{2}-
\angle_{q^{k}_{0}}(T_{k}(q^{k}_{0}),\overline{q^k_0 p})\Big)$$
$$= \pi-
\Big(\angle_{q^i_1}(T_i(q^i_1),\overline{q^i_1 p}) +
\angle_{q^k_0}(T_{k}(q^k_0),\overline{q^k_0 p})\Big)$$
$$= \pi-\angle_{q^i_1}(T_i(q^i_1),T_{k}(q^k_0)),$$
which is the exterior angle between $c_i$ and $c_k$ at
$q^i_1 \sim q^k_0$, relative to the constant-curvature cone
$\widehat C'$.

Therefore
\begin{eqnarray} \label{eqn:GB for widehat C}
2\pi\Theta(\widehat{C}',p)=
-\kappa^2\area(\widehat{C}')&-&\sum_{i=1}^{2m}\int_{c'_i}
\overrightarrow{k}\cdot\nu_{\widehat{c}'_{i}}\,ds \\
&+&\sum_{i=1}^{2m}\sum_{j=0,1}\Big(\frac{\pi}{2}-\angle_{q^{i}_{j}}(T_{i}(q^{i}_{j}),\overline{q^{i}_{j}p})\Big).
\nonumber
\end{eqnarray}
 For the edges $c_{i_1}$ and
$c_{i_2}$ which represent the same edge $c_i$ of $\Gamma$, we have
\begin{eqnarray*}
\int_{\widehat{c}_{i}} \overrightarrow{k}\cdot\nu_{\widehat{C}}ds=
\int_{\widehat{c}_{i_1}}
\overrightarrow{k}\cdot\nu_{\widehat{C}}ds
=\int_{\widehat{c}_{i_2}}
\overrightarrow{k}\cdot\nu_{\widehat{C}}ds.
\end{eqnarray*}
We also note that the exterior angle appears twice and
$\area(\widehat{p}\cone \widehat{\Gamma}'
)=2\area(\widehat{p}\cone \widehat{\Gamma})$. Therefore by
dividing both sides of equation \eqref{eqn:GB for widehat C}
by two, we obtain
$$2\pi\Theta(\widehat{C},p)+\kappa^{2}\area(\widehat{C})=
-\sum_{i=1}^m \int_{\widehat{c}_{i}}
\overrightarrow{k}\cdot\nu_{\widehat{c}_{i}}ds
+\sum_{i=1}^m\sum_{j=0,1}\Big(\frac{\pi}{2}
-\angle_{q^{i}_{j}}(T_{i}(q^{i}_{j}),\overline{q^{i}_{j}p})\Big),$$
which completes the proof.
\end{proof}
\begin{df}
{\rm Define the} minimum cone area $\mathcal{{A}}(\Gamma)$ {\rm of a graph}
$\Gamma \subset M$ {\rm as}
$$\mathcal{{A}}(\Gamma):= \inf_{p\in \cv(\Gamma)}\area({p\cone\Gamma})$$
\end{df}
\begin{thm} \label{thm:3pi}
Let $M$ be an $n$-dimensional complete simply connected Riemannian
manifold whose sectional curvature is bounded above by a
nonpositive constant $-\kappa^2$. Let $\Gamma$ be a graph in $M$
with $TC(\Gamma)\leq 3\pi+\kappa^{2}\mathcal{{A}}(\Gamma)$ and let
$\Sigma \in \mathcal{S}_\Gamma$ be a strongly stationary surface
with respect to $\Gamma$ in $M$. Then $\Sigma$ is either an
embedded surface or a subset of the $Y$-singularity cone with
totally geodesic faces having constant Gauss curvature $-\kappa^2$.
\end{thm}
\begin{proof}
For a point $p$ on $\Sigma \setminus \Gamma$, Proposition
\ref{prop:density in M} says that
\begin{eqnarray*}\label{eq:b}
2\pi\Theta(\Sigma,p) &\leq& 2\pi\Theta(\widehat{C},p).
\end{eqnarray*}
Applying Proposition \ref{prop:CG1}, Proposition \ref{prop:CG2},
and Proposition \ref{prop:GB for widehat of C}, we have
\begin{eqnarray*}
2\pi\Theta(\widehat{C},p) &=& -\kappa^{2}
\area(\widehat{C})-\int_{\Gamma} \widehat{k}\, ds +
\sum_{i =1}^{m}\sum_{j=0,1}\Big(\frac{\pi}{2}-
\angle_{q^{i}_{j}}(T_{i}(q^{i}_{j}),\overline{q^{i}_{j}p})\Big)\\
&\leq& -\kappa^{2} \area(p\cone \Gamma)-\int_{\Gamma} {k}\, ds
+\sum_{i=1}^{m}\sum_{j=0,1}\Big(\frac{\pi}{2}-
\angle_{q^{i}_{j}}(T_{i}(q^{i}_{j}),\overline{q^{i}_{j}p})\Big)\\
&\leq& TC(\Gamma)-\kappa^{2}\area(p\cone\Gamma).
\end{eqnarray*}
Therefore the assumption on the cone total curvature of $\Gamma$
implies that
\begin{eqnarray*}
2\pi\Theta(\Sigma,p) \leq 3\pi.
\end{eqnarray*}
If $\Theta(\Sigma,p)<\frac{3}{2}$ for any $p\in \Sigma \setminus \Gamma$, then $\Sigma$ is an embedded
surface. Otherwise, we have $\Theta(\Sigma,p)=\frac{3}{2}$ which
means that $\Sigma$ is a cone with vertex $p$ and totally geodesic
faces. Since the only stationary cone with density $\frac{3}{2}$
is the $Y$-cone, $\Sigma$ is a subset of the $Y$-cone.
\end{proof}
\begin{thm}\label{thm:2pi}
Let $M^3$ be a $3$-dimensional complete simply connected
Riemannian manifold whose sectional curvature is bounded above by
a nonpositive constant $-\kappa^2$. Let $\Sigma \in
\mathcal{S}_{\Gamma}$ be embedded as an
$(\mathbf{M},0,\delta)$-minimizing set with respect to a graph
$\Gamma$ in $M^3$. If $TC(\Gamma)\leq 2\pi
C_{T}+\kappa^{2}\mathcal{{A}}(\Gamma)$, then $\Sigma$ can have
only $Y$ singularities unless $\Sigma$ is a subset of the $T$
stationary cone with totally geodesic faces.
\end{thm}
\begin{proof}
As in the proof of the above theorem, we can see that  at $p \in
\Sigma \setminus \Gamma$
\begin{eqnarray*}
2\pi\Theta(\Sigma,p)&\leq& 2\pi\Theta(\widehat{C},p)\\
&\leq& TC(\Gamma)-\kappa^{2}\area(p\cone\Gamma)\leq 2\pi C_{T}.
\end{eqnarray*}
By Theorem \ref{Taylor:cone}, we see that for $p\in \Sigma
\setminus \Gamma$, if $\Theta(\Sigma,p)=1$ then $\Sigma$ is a
plane. If $\Theta(\Sigma,p)=\frac{3}{2}$ then $\Sigma$ is a subset
of the $Y$-cone. If $\Theta(\Sigma,p)=C_{T}$ then $\Sigma$ is a
subset of the $T$-cone.
\end{proof}

%
%
\section{Regularity of soap film-like surfaces in spaces with
bounded diameter}

Let $M^n$ be a simply-connected Riemannian manifold with diameter
$\leq \frac{\pi}{b}$ and sectional curvatures $K_M\leq b^2$, where
$b>0$. Recall that any two points of $M$ are connected by a unique
geodesic in $M$. As in section 3, we may define the {\em convex hull}
$\cv(S)$ of a subset $S$ of $M$. It is not difficult to see that
$\cv(S)$ is the intersection of closed geodesically convex sets
containing $S$, and it follows that if $\Sigma \in \mathcal{S}_\Gamma$
is strongly stationary with respect to a
graph $\Gamma$ in $M^n$, then $\Sigma \subset \cv(\Gamma)$.

\begin{df}
Let $\widehat{g}$ be a new metric on $C=p\cone \Gamma$ with constant
Gauss curvature $b^2$ such that the distance from $p$
remains the same as in the original metric $g$, and so does the
arclength element of $\Gamma$.
\end{df}
As in section 3, every geodesic from $p$
under $g$ remains a geodesic of equal length under $\widehat{g}$,
the length of any arc of $\Gamma$ remains the same, and the angles
between the tangent vector to $\Gamma$ and the geodesic from $p$
remain unchanged. $\widehat C = (C, \widehat g)$ may be
constructed as in Section 3, with the two-dimensional sphere of
radius $1/b$ replacing $\mathbb{H}^2(-\kappa^2)$.

\begin{prop} \label{prop:density in the sphere}
Let $\Sigma \in \mathcal{S}_\Gamma$ be a strongly stationary surface in
$M^n$ and let $p$ be a point in $\Sigma \setminus \Gamma$. Then
$$\Theta(\Sigma,p)< \Theta(\widehat C,p),$$
unless $\Sigma$ is a cone over $p$ with totally
geodesic faces and constant Gauss curvature $b^2$.
\end{prop}
\begin{proof}
The proof involves computations analogous to Proposition
\ref{prop:density in M} on $\Sigma$ and on
$\widehat C$, related to the well-known monotonicity inequality.

Let $G(r(x))=\log \tan(br(x)/2)$ be the Green's function of the
two-dimensional sphere of constant curvature $b^2$, where $r(x)$
is the distance from $p$ in
$M^n$. On an immersed minimal surface $\Sigma_i$ in
$M$, it follows from \cite{CG0} that
\begin{eqnarray}\label{green ftn on S}
\triangle_{\Sigma_i} G(r) = 2\frac{b^2\cos br}{\sin^{2} br}
(1-|\nabla_{\Sigma_i}r|^{2}) \geq 0.
\end{eqnarray}
Recall that $\Sigma = \cup_{i \in I} \Sigma_i \in
\mathcal{S}_\Gamma$ and each $\Sigma_i$ is minimal. Integrating
(\ref{green ftn on S}) over each $\Sigma_i \setminus
B_\varepsilon(p)$ for sufficiently small $\varepsilon >0$ gives
\begin{eqnarray*}
0\leq \int_{\Sigma_i \setminus
B_\varepsilon(p)}\triangle_{\Sigma_{i}}G\,dA
=\int_{\partial(\Sigma_i \setminus B_\varepsilon(p))}\frac{b}{\sin
br}\frac{\partial r}{\partial \nu_{\Sigma_i}}ds.
\end{eqnarray*}
Since $\partial(\Sigma_i \setminus
B_\varepsilon(p))=(\partial\Sigma_{i}\cap\Gamma)\cup (\partial
B_{\varepsilon}(p)\cap \Sigma_{i})\cup (\partial
\Sigma_{i}\setminus (\Gamma\cup \overline{B_{\varepsilon}(p)}))$,
summing over $i$ gives

\begin{eqnarray*}
0\leq \int_{\Gamma}\frac{b}{\sin br}\frac{\partial r}{\partial\nu_{\Sigma}}ds
+\int_{\Sigma \cap \partial B_{\varepsilon}(p)}\frac{b}{\sin
b\varepsilon}\frac{\partial r}{\partial\nu_{\Sigma}}ds +\sum_{i\in
I}\int_{\partial\Sigma_{i}\setminus(\Gamma\cup\overline{B_{\varepsilon}(p)})}
\frac{b}{\sin br}\frac{\partial r}{\partial\nu_{\Sigma_{i}}}ds .
\end{eqnarray*}
Applying the balancing condition (\ref{balancing}) at each point $p\in \partial \Sigma_i \setminus \Gamma$, one can see
that the last term vanishes.

Since along $\Sigma\cap \partial B_{\varepsilon}(p),$
$\frac{\partial r}{\partial\nu_{\Sigma}}\rightarrow -1$ uniformly
as $\varepsilon\rightarrow 0$, the second term converges to
$$\lim_{\varepsilon \rightarrow 0}-\frac{b\,\length (\Sigma \cap
\partial B_{\varepsilon}(p))}{ \sin b\varepsilon} =
-2\pi\Theta(\Sigma,p).$$

Hence we have
\begin{equation}\label{eq:a11}
2\pi\,\Theta(\Sigma,p)\leq \int_{\Gamma}
\frac{b}{\sin br}\frac{\partial r}{\partial\nu_{\Sigma}}ds.
\end{equation}

Similarly, one can estimate the density $\Theta (\widehat C,p)$ at
$p$ of the cone $p\cone \Gamma$ with the metric $\widehat g$ of
constant Gauss curvature $b^2$. Because
$\Gamma = \cup_{j=1}^m c_j$, where each arc $c_j$ is $C^2$ and $C^1$
up to the end points, the cone $C=p\cone \Gamma$ can be represented as
\begin{eqnarray*}
C=p\cone \Gamma = \cup \overline{A_j} =
\cup (A_j\cup\partial A_j),
\end{eqnarray*}
where $A_j=p\cone c_j$ is a $C^2$ surface, $C^1$ up to its
boundary, possibly not immersed. Write
$\widehat C=\cup \overline{\widehat A_j}$.
On each $\widehat A_j$, it follows from \cite{CG0} that
the Green's function $G(r)$ is harmonic because
\begin{eqnarray} \label{eq:harmonic}
\triangle_{\widehat C} G(r) = 2\frac{b^2\cos br}{\sin^{2} br}
(1-|\nabla_{\widehat C}r|^{2}) =0.
\end{eqnarray}
The divergence theorem yields
\begin{eqnarray} \label{eqn:cone minus ball}
0= \int_{A_{j}\setminus B_{\varepsilon}(p)}\triangle_{\widehat C}G\,dA
=\int_{\partial (A_{j}\setminus B_{\varepsilon}(p))}\frac{b}{\sin
br}\frac{\partial r}{\partial \nu_{\widehat C}} ds.
\end{eqnarray}
As before, each boundary $\partial (A_{j}\setminus B_{\varepsilon}(p))$
consists of three parts.

Summing the above equation (\ref{eqn:cone minus ball}) over
$j=1,\cdots,m$, since $\nu_{\widehat C}= \nu_C$ along $\Gamma$, we have
$$0= \int_{\Gamma}\frac{b}{\sin br}\frac{\partial
r}{\partial\nu_{\widehat C}}ds
+\int_{C \cap \partial B_{\varepsilon}(p)}\frac{b}{\sin
b\varepsilon}\frac{\partial r}{\partial\nu_{\widehat C}}ds
+\sum_{j}\int_{(p\cone \partial c_j) \setminus B_{\varepsilon}(p)}
\frac{b}{\sin br}\frac{\partial r}{\partial\nu_{\widehat A_{j}}}ds,$$
where
$\nu_{A_{j}}$ denotes the outward unit conormal vector along the
boundary $\partial A_{j}$ and $\nu_{C}(q)$ is defined to be
$\sum_{j=1}^{m}\{\nu_{A_{j}}(q):q\in \partial A_{j}\}$.
Note that along $p\cone(\partial c_{j})$, the conormal vector
$\nu_{{\widehat A_{j}}}$ and $\overline{\nabla} r$ are perpendicular, i.e.
$\frac{\partial r}{\partial\nu_{\widehat A_j}}\equiv 0$. Hence the last
term vanishes in the above equation. Since
$\frac{\partial r}{\partial \nu_{\widehat C}} \rightarrow 1$ as
$\varepsilon \rightarrow 0$ along
$C\cap\partial B_\varepsilon (p)$ and hence
\begin{eqnarray*}
-\frac{b\,\length (C\cap B_{\varepsilon}(p))}{\sin{b\varepsilon}}
\rightarrow -2\pi\, \Theta(C,p),
\end{eqnarray*}
one can see that the second term converges to $-2\pi \Theta(C,p)$.
Therefore we have
\begin{equation}\label{eq:a22}
 2\pi\Theta(C,p)=
\int_{\Gamma}\frac{b}{\sin br}\frac{\partial r}{\partial\nu_{C}}ds.
\end{equation}
Since $\frac{\partial r}{\partial \nu_{C}}\geq\frac{\partial
r}{\partial\nu_{\Sigma}}$ almost everywhere along $\Gamma$, it
follows from (\ref{eq:a11}) and (\ref{eq:a22}) that
$$2\pi\Theta(\Sigma,p)\leq
\int_{\Gamma}\frac{b}{\sin br}\frac{\partial
r}{\partial\nu_{\Sigma}}ds\leq \int_{\Gamma}\frac{b}{\sin
br}\frac{\partial r}{\partial\nu_{C}}ds=2\pi\Theta(C,p).$$
If equality holds, then $\Delta_{\Sigma} G \equiv0$. So we should
have $|\nabla_{\Sigma}r|\equiv 1$ by (\ref{eq:harmonic}), which
can happen only when $\Sigma$ is totally geodesic; and
$\triangle_\Sigma G \equiv 0$ \cite{CG0}.
\end{proof}

We prove the following Gauss-Bonnet formula for singular
two-dimensional cones in $M^n$.
%
\begin{prop}[Gauss-Bonnet formula] \label{prop:GB in S}
Let $p$ be a point of $M^n\setminus \Gamma$ for a
graph  $\Gamma =\cup c_i$. Let $\widehat{C}=(C, \widehat{g})$ for
$C=p\cone\Gamma$.  Then we have 
$$2\pi\Theta(\widehat{C},p)=
b^2 \area (\widehat C) -\sum_{i=1}^{m}\int_{c_{i}}
\overrightarrow{k}\cdot\nu_{c_{i}}ds +
\sum_{i=1}^{m}\sum_{j=0,1}(\frac{\pi}{2}-
\angle_{q^{i}_{j}}(T_{i}(q^{i}_{j}),\overline{q^{i}_{j}p}))$$
where $\overrightarrow{k}$ is the curvature vector of $c_{i}$ in
${\widehat C}$, $\nu_{c_{i}}$ is the outward unit conormal
vector to $p\cone\Gamma$ along $c_{i}$,
$\angle_{q^{i}_{j}}(T_{i}(q^{i}_{j}),\overline{q^{i}_{j}p})$ is
the angle at $q^{i}_{j}$ between the tangent vector
$T_{i}(q^{i}_{j})$ to $c_{i}$ and the geodesic
$\overline{q^{i}_{j}p}$.
\end{prop}

\begin{pf}
Similar arguments as in the proof of Proposition 
\ref{prop:GB for widehat of C} give the desired result. \qed 
\end{pf}

As a consequence of the density comparison (Proposition
\ref{prop:density in the sphere}) and the Gauss-Bonnet Theorem for
two-dimensional cones in $M^n$ (Proposition
\ref{prop:GB in S}), there follows
\begin{thm}
Let $M^n$ be a manifold with $K_M \leq b^2$ and diameter
$\leq \frac{\pi}{b}$. Let $\Gamma$ be a graph in $M^n$ and let $\Sigma
\in \mathcal{S}_\Gamma$ be a strongly stationary surface with
respect to $\Gamma$ in  $M^n$. Then for $p\in M^n$ we have
\begin{eqnarray*}
2\pi \Theta (\Sigma,p) \leq TC(\Gamma) +b^2\area(p\cone\Gamma).
\end{eqnarray*}
\end{thm}

%
\begin{df}
{\rm Define the} maximum cone area
$\mathcal{\overline{A}}(\Gamma)$ {\rm of a graph} $\Gamma \subset
M^n$ {\rm as}
$$\mathcal{\overline{A}}(\Gamma):=
\sup_{p\in \cv(\Gamma)}\area({p\cone\Gamma}).$$
\end{df}

Now we state and prove the regularity theorems for soap film-like
surfaces spanning graphs with small cone total curvature in
a manifold of bounded diameter.
%
\begin{thm}\label{only Y}
Let $M^n$ be a manifold with $K_M \leq b^2$ and diameter
$\leq \frac{\pi}{b}$, and let $\Gamma$ be a graph with cone total
curvature
$TC(\Gamma)\leq 3\pi-b^2\mathcal{\overline{A}}(\Gamma)$.
Let $\Sigma \in \mathcal{S}_\Gamma$ be a strongly stationary
surface with respect to $\Gamma$ in  $M^n$. Then
$\Sigma$ is either an embedded surface or a subset of the
$Y$-singular cone formed by three totally geodesic surfaces of
constant Gauss curvature $b^2$.
\end{thm}

\begin{proof}
At $p \in \Sigma \setminus \Gamma$, we have
$$2\pi\Theta(\Sigma,p)\leq2\pi\Theta(p\cone\Gamma,p)
\leq TC(\Gamma)+ b^2 \area(p\cone\Gamma)\leq 3\pi.$$
If $\Theta(\Sigma,p)<\frac{3}{2}$, then $\Sigma$ is an embedded
surface. Otherwise, $\Theta(\Sigma,p)=\frac{3}{2}$ which means
that $\Sigma$ is a cone with vertex $p$ and totally geodesic
faces. Since the only stationary cone with density $\frac{3}{2}$
is the $Y$-cone, $\Sigma$ is a subset of the $Y$-cone.
\end{proof}
Given a strongly stationary surface $\Sigma$ in
$\mathcal{S}_\Gamma$, we have seen that the first nontrivial upper
bound for the density of $\Sigma$, other than $1$, is $3/2$. In
order to find a larger upper bound for density, we consider the
case of ambient dimension $n=3$.
Note that the tangent cone of a strongly stationary
soap film-like surface in $M^3$ is exactly the same as in 
Euclidean space $\mathbb{R}^3$. Since
there are only three area minimizing cones in $\mathbb{R}^3$
\cite{AT}, the only possible candidate for a larger bound for
density is the $T$-singularity cone.

%
\begin{thm}\label{only T}
Let $M^n$ be a manifold with $K_M \leq b^2$ and diameter
$\leq \frac{\pi}{b}$. Let $\Gamma$ be a graph in $M^3$ with
$TC(\Gamma)\leq 2\pi C_{T}-b^2\mathcal{\overline{A}}(\Gamma)$ and
let $\Sigma \in \mathcal{S}_{\Gamma}$ be embedded as an
$(\mathbf{M},0,\delta)$-minimizing set with respect to $\Gamma$.
Then $\Sigma$ can have only $Y$ singularities unless $\Sigma$ is a
subset of the $T$ stationary cone with totally geodesic faces.
\end{thm}
\begin{proof}
For any point $p \in \Sigma \setminus \Gamma$ we have
$$2\pi\Theta(\Sigma,p)\leq2\pi\Theta(p\cone\Gamma,p)
\leq TC(\Gamma)+b^2\area(p\cone\Gamma)\leq 2\pi C_{T}.$$
It follows from Theorem \ref{Taylor:cone} that the tangent cone of
an $(\mathbf{M}, 0,\delta)$-minimal set $\Sigma$ at $p$ is
area-minimizing with respect to the intersection with the unit
sphere centered at $p$, and that the plane, the $Y$-cone and the
$T$-cone are the only possibility for a area-minimizing tangent
cone. Hence for $p\in \Sigma \setminus \Gamma$, if
$\Theta(\Sigma,p)=1$ then $\Sigma$ is a plane. If
$\Theta(\Sigma,p)=\frac{3}{2}$ then $\Sigma$ is the $Y$-cone.
Moreover if $\Theta(\Sigma,p)=C_{T}$ then $\Sigma$ is the
$T$-cone.
\end{proof}

\begin{Rmk}
We do not know whether in general an
$(\mathbf{M}, 0, \delta)$-minimizing set
with respect to a graph is an element of the class
$\mathcal{S}_{\Gamma}$ or not.  Note that an
$(\mathbf{M},0,\delta)$-minimal set $\Sigma$ in
$\mathcal{S}_\Gamma$ with variational boundary $\Gamma$ is
strongly stationary with respect to $\Gamma$. However the converse
is not true in general. For instance, consider the cone over the
$1$-skeleton $\Gamma$ of the cube. It is strongly stationary with
respect to $\Gamma$, but not an $(\mathbf{M},0,\delta)$-minimal
set. (See \cite{Taylor}.)
\end{Rmk}

For our final theorem, we return to the context of minimal branched
immersions of surfaces, as treated in \cite{EWW} and in \cite{CG}.
Note that the density of such a surface $\Sigma$ in $M^n$ must be
an integer $\geq 1$ at each point. At a branch point, the density 
is $\geq 2$; at a point $p\in M$ of self-intersection, the density
$\Theta(\Sigma,p)$ equals the number of pieces of surface which
intersect at $p$.  In the same spirit as the theorems above, which
refer to
$C_Y = \frac{3}{2}$ and $C_T\approx 1.8245$, we may define 
$C_X = 2$, the minimum density at a self-intersection point or
branch point of a branched immersion.

%
\begin{thm}
Let $M^n$ be a manifold with $K_M \leq b^2$ and diameter
$\leq \frac{\pi}{b}$. Let $\Gamma$ be a simple closed curve in 
$M^3$ with
$TC(\Gamma) < 2\pi C_X-b^2\mathcal{\overline{A}}(\Gamma)$, and
let $\Sigma$ be a branched immersion of a compact surface into
$M^n$ with boundary $\Gamma$.
Then $\Sigma$ is embedded.
\end{thm}

The proof is analogous to the proofs of Theorems \ref{only Y} 
and \ref{only T}, showing that the density of $\Sigma$ at
any point of the convex hull of $\Gamma$ is less than two. This
result is Theorem 1 of \cite{CG} in the specific case where $M$ 
is the $n$-dimensional hemisphere of constant sectional curvature 
$b^2$.

\vspace{1cm}
\ \  \\
\noindent Robert Gulliver\\
School of Mathematics, University of Minnesota,
Minneapolis, MN 55455, USA\\
{\tt e-mail:gulliver@math.umn.edu}\\

\noindent Sung-Ho Park\\
School of Mathematics, Korea Institute for Advanced Study, 207-43, Cheongnyangni 2-dong, Dongdaemun-gu, Seoul 130-722, Korea\\
{\tt e-mail:shubuti@kias.re.kr}\\

\noindent Juncheol Pyo\\
Department of Mathematics, Seoul National University, Seoul 151-742, Korea\\
{\tt e-mail:jcpyo@snu.ac.kr}\\

\noindent Keomkyo Seo\\
Department of Mathematics, Sookmyung Women's University, Hyochangwongil 52, Yongsan-ku, Seoul 140-742, Korea\\
e-mail : kseo@sookmyung.ac.kr
{\tt e-mail:kseo@kias.re.kr}


\begin{thebibliography}{123}
\bibitem{Almgren} F. Almgren, {\em Existence and regularity almost
everywhere of solutions to elliptic variational problems with
constraints}, Mem. Amer. Math. Soc. {\bf 4} (1976), no. 165.
\bibitem{AT} F. Almgren and J. Taylor, {\em Geometry of soap
films}, Scientific American. {\bf 235} (1976), 82--93.
\bibitem{Choe} J. Choe, {\em The isoperimetric inequality for
minimal surfaces in a Riemannian manifold}, J. reine angew. Math.
{\bf 506} (1999), 205--214.
\bibitem{CG0} J. Choe and R. Gulliver, {\em Isoperimetric
inequalities on minimal submanifolds of space forms}, Manuscripta
Math. {\bf 77} (1992), 169--189.
\bibitem{CG} J. Choe and R. Gulliver, {\em Embedded minimal
surfaces and curvature of curves in a manifold}, Math. Res. Lett.
{\bf 10} (2003), 343--362.
\bibitem{EWW} T. Ekholm, B. White, and D. Wienholtz, {\em
Embeddedness of minimal surfaces with total boundary curvature at
most $4\pi$}, Ann. of Math. {\bf 155} (2002), 209--234.
\bibitem{Federer} H. Federer, {\em Geometric measure theory},
Springer-Verlag, New York, 1969.
\bibitem{GY} R. Gulliver and S. Yamada, {\em Area density and
regularity for soap film-like surfaces spanning graphs}, Math. Z.
{\bf 253} (2006), 315--331.
\bibitem{KNS} D. Kinderlehrer, L. Nirenberg and J. Spruck, {\em
Regularity in elliptic free boundary problems, I}, J. Amer. Math.
Soc. {\bf 34}, 86--119 (1978).
\bibitem{Ore} O. Ore, {\em Graphs and their uses}, Random House,
New York, 1963.
\bibitem{Simon} L. Simon, {\em Lectures on geometric measure
theory}, Proc. Centre Math. Anal. Austral. Nat. Univ. Vol. {\bf
3}, Canberra, Australia, 1983.
\bibitem{Taylor} J. Taylor, {\em The structure of singularities in
soap-bubble-like and soap-film-like minimal surfaces}, Ann. of
Math. {\bf 103} (1976), 489--539.

\end{thebibliography}
\end{document}